\documentclass{amsart}

\usepackage{amsmath, amsfonts, amssymb}
\usepackage{graphicx}
\usepackage{psfrag}

\newtheorem{thm}{Theorem}[section]
\newtheorem{lem}[thm]{Lemma}
\newtheorem{prop}[thm]{Proposition}

\newtheorem{defn}[thm]{Definition}

\newcommand{\N}{\mathbb{N}}
\newcommand{\Z}{\mathbb{Z}}
\newcommand{\llangle}{\langle\!\langle}
\newcommand{\rrangle}{\rangle\!\rangle}
\newcommand{\bast}{{\displaystyle\ast}}
\newcommand{\mc}[1]{\mathcal{#1}}
\newcommand{\fm}[1]{\mc{#1}^{\pm\bast}}
\newcommand{\ol}[1]{\overline{#1}}

\providecommand{\Area}{\mathop{\rm Area} \nolimits}
\providecommand{\FreeEq}{\mathop{\stackrel{\rm free}{=}} \nolimits}
\providecommand{\Dist}{\mathop{\rm Dist} \nolimits}
\providecommand{\RArea}{\mathop{\rm RArea} \nolimits}
\providecommand{\Rad}{\mathop{\rm Rad} \nolimits}

\providecommand{\Edge}{\mathop{\rm Edge} \nolimits}
\renewcommand{\Vert}{\mathop{\rm Vert} \nolimits}

\title{An Isoperimetric Function for Bestvina-Brady Groups}
\author{Will Dison}

\begin{document}

\begin{abstract}
Given a right-angled Artin group $A$, the associated Bestvina-Brady
group is defined to be the kernel of the homomorphism $A \rightarrow
\Z$ that maps each generator in the standard presentation of $A$ to
a fixed generator of $\Z$.  We prove that the Dehn function of an
arbitrary finitely presented Bestvina-Brady group is bounded above
by $n^4$.  This is the best possible universal upper bound.
\end{abstract}

\maketitle

\section{Introduction}

Dehn functions and right-angled Artin groups are some of the most
studied objects in contemporary geometric group theory.  Among the
most striking works concerning right-angled Artin groups is the
combinatorial Morse theory introduced by Bestvina and Brady in
\cite{Bestvina1} to solve the long-standing question of whether there
exist groups of type FP(2) which are not finitely presented.  The
central objects of study in their theory are the Bestvina-Brady
groups, which arise as the kernels of homomorphisms from
right-angled Artin groups to the integers.

A finite flag simplicial complex $\Delta$ with vertices $v_1,
\ldots, v_k$ defines a right-angled Artin group $A$ with
presentation \[\mc{P}_A = \langle a_1, \ldots, a_k \, | \, [a_i,
a_j] \text{ whenever $v_i$ and $v_j$ are joined by an edge in
$\Delta$} \rangle.\] The \emph{Bestvina-Brady group} $H_\Delta$
associated to $\Delta$ is defined to be the kernel of the
homomorphism $A \rightarrow \Z = \langle t \rangle$ which maps each
$a_i \mapsto t$. In \cite{Bestvina1} the authors prove that the group
$H_\Delta$ is finitely presented if and only if $\Delta$ is simply
connected.  The purpose of this article is to  estimate the
complexity of the word problem in Bestvina-Brady groups by
establishing a universal upper bound on their Dehn functions.

\begin{thm} \label{thm1}
  If $\Delta$ is simply connected then the Dehn function $\delta$ of
  $H_\Delta$ satisfies $\delta(n) \preceq n^4$.
\end{thm}

This result is sharp: there exist finitely presented Bestvina-Brady
groups whose Dehn functions are $\simeq n^4$ (see \cite{Brady1}).
Theorem \ref{thm1} provides an obstruction to the method suggested
in \cite{Brady1} for producing Bestvina-Brady groups whose Dehn
functions are similar to $n^k$ for arbitrary integers $k$.

A significant component of the proof of Theorem~\ref{thm1} is a
method for producing an isoperimetric function $f$ for a finitely
presented group $K$ from an isoperimetric function for a cyclic
extension of $K$. \textit{A priori}, the function $f$ will be an
isoperimetric function for a presentation of $K$ with infinitely
many relators.  We introduce the notion of area-penetration pairs to
deal with such non-finite presentations and show how they can be
used to derive an isoperimetric function for a finite presentation
of a group from an isoperimetric function for a presentation of the
group with infinitely many relators.

The organisation of this paper is as follows.  Section 2 begins with
the definitions of various filling invariants for finitely generated
groups; namely, isoperimetric and Dehn functions and area-radius
pairs.  We then introduce the new notions of area-penetration pairs
and relative area functions, which we will use to deal with
presentations with infinitely many relators.  In Section~3 we prove
a general result concerning the isoperimetric functions of cyclic
extensions. Theorem \ref{thm1} is proved as a corollary of this in
Section~4. Finally, in Section~5 we briefly recount the construction
due to Brady, Forester and Shankar of a finitely presented
Bestvina-Brady group with Dehn function $\simeq n^4$.

\section{Filling Functions}

In this section we define various filling invariants of groups and
give some of their basic properties.  Throughout $\mc{P} = \langle
\mc{A} \, | \, \mc{R} \rangle$ will be a presentation with $\mc{A}$
finite.

\subsection{Area Functions}

We recall the basic definitions concerning isoperimetric functions
for finitely generated groups.  For further background and a more
thorough exposition see, for example, \cite{brid02} or
\cite{Riley1}.  Note that the definitions given here are standard,
but we do not make the usual assumption that the presentations
involved have a finite number of relators.

Given a set $\mc{A}$, write $\mc{A}^{-1}$ for the set of formal
inverses of the elements of $\mc{A}$ and write $\mc{A}^{\pm1}$ for
the set $\mc{A} \cup \mc{A}^{-1}$.  Denote by $F(\mc{A})$ the free
group on the set $\mc{A}$ and by $\mc{A}^{\pm\bast}$ the free monoid
on the set $\mc{A}^{\pm1}$.  We refer to elements of
$\mc{A}^{\pm\bast}$ as words in the letters $\mc{A}^{\pm1}$ and
write $\emptyset$ for the empty word. The length of a word $w \in
\mc{A}^{\pm\bast}$ is written $|w|$. Given words $w_1, w_2 \in
\fm{A}$ we write $w_1 \FreeEq w_2$ if $w_1$ and $w_2$ are equal as
elements of $F(\mc{A})$ and $w_1 \equiv w_2$ if $w_1$ and $w_2$ are
equal as elements of $\fm{A}$.

\begin{defn} \label{def1}
  A word $w \in \mc{A}^{\pm\bast}$ is said to be
  \emph{null-homotopic} over $\mc{P}$ if it represents the
  identity in the group presented by $\mc{P}$.  A \emph{null-$\mc{P}$-expression} for such a
  word is a sequence $(x_i, r_i)_{i=1}^m$ in $\fm{A} \times
  \mc{R}^{\pm1}$ such that \[w \FreeEq \prod_{i=1}^m x_i r_i
  x_i^{-1}.\]  Define the \emph{area} of a null-$\mc{P}$-expression $\Sigma$, written $\Area
  \Sigma$, to be the integer $m$.  Define the \emph{$\mc{P}$-Area} of $w$,
  written $\Area_\mc{P}(w)$, to be the minimal area taken over all
  null-$\mc{P}$-expressions for $w$.

  The Dehn function of the presentation $\mc{P}$, written $\delta_\mc{P}$,
  is defined to be the function $\N \rightarrow \N$ given by \[\delta_\mc{P}(n)
  = \max \{ \Area_\mc{P}(w) \, : \, \text{$w \in \fm{A}$, $w$ null-homotopic,
  $|w| \leq n$} \}.\]
\end{defn}

Although the Dehn functions of different \emph{finite} presentations
of a fixed group may differ, their asymptotic behaviour will be the
same.  This is made precise in the following way.

\begin{defn}
  Given functions $f, g : \N \rightarrow \N$ write $f \preceq g$ if
  there exists a constant $C > 0$ such that $f(n) \leq Cg(Cn + C) +
  Cn + C$ for all $n$.  Write $f \simeq g$ if $f \preceq g$ and
  $g \preceq f$.
\end{defn}

If $\mc{P}_1$ and $\mc{P}_2$ are finite presentations of the same
group then $\delta_{\mc{P}_1} \simeq \delta_{\mc{P}_2}$ (see, e.g.,
\cite{brid02}).

\begin{defn}
  A function $f : \N \rightarrow \N$ is an \emph{isoperimetric
  function} for a group $G$ if $\delta_\mc{Q} \preceq f$ for some (and hence
  any) finite presentation $\mc{Q}$ of $G$.
\end{defn}

\begin{defn}
  A \emph{null-$\mc{P}$-scheme} for a null-homotopic word $w \in
  \mc{A}^{\pm\bast}$ is a sequence $w \equiv w_0 \rightsquigarrow w_1
  \rightsquigarrow \ldots \rightsquigarrow w_n \equiv \emptyset$ of words in
  $\mc{A}^{\pm\bast}$ such that each $w_i w_{i+1}^{-1}$ is
  null-homotopic.  The \emph{$\mc{P}$-Cost} of each transition $w_i
  \rightsquigarrow w_{i+1}$ is the $\mc{P}$-Area of the word $w_i
  w_{i+1}^{-1}$.
\end{defn}

Note that the sum of the costs of the transitions in a
null-$\mc{P}$-scheme gives an upper bound on the area of the word
$w$.

\subsection{Area-Radius pairs}

\begin{defn}
  Define the \emph{radius} of a null-$\mc{P}$-expression $\Sigma = (x_i,
  r_i)_{i=1}^m$, written $\Rad \Sigma$, to be $\max_{i=1}^m
  |x_i|$.  A pair $(\alpha, \rho)$ of functions $\alpha, \rho : \N \rightarrow \N$
  is said to be an \emph{area-radius pair} for the presentation $\mc{P}$ if
  for all null-homotopic words $w \in \fm{A}$
  with $|w| \leq n$ there exists a null-$\mc{P}$-expression $\Sigma$ with $\Area \Sigma
  \leq \alpha(n)$ and $\Rad \Sigma \leq \rho(n)$.
\end{defn}

The following result shows how area-radius pairs transform under
change of presentation.

\begin{prop} \label{prop4}
  Let $\mc{P}$ and $\mc{Q}$ be finite presentations of the same
  group.  If $(\alpha, \rho)$ is an area-radius pair for $\mc{P}$
  then there exists an area-radius pair $(\alpha', \rho')$ for
  $\mc{Q}$ with $\alpha \simeq \alpha'$ and $\rho \simeq \rho'$.
\end{prop}

\begin{proof}
  Since $\mc{P}$ can be converted to $\mc{Q}$ by a finite sequence
  of Tietze transformations, it suffices to prove the proposition in
  the situation that $\mc{P}$ and $\mc{Q}$ are related by a single
  such transformation.  There are four cases to consider.

  \textbf{Case 1.}  Suppose that $\mc{P} = \langle \mc{A} \, | \,
  \mc{R} \rangle$ and $\mc{Q} = \langle \mc{A} \, | \, \mc{R}, s
  \rangle$ where $s \in \fm{A}$ is null-homotopic over $\mc{P}$.
  A null-$\mc{P}$-expression for a word $w \in \fm{A}$ is
  also a null-$\mc{Q}$-expression for $w$, so $(\alpha, \rho)$ is
  itself an area-radius pair for $\mc{Q}$.

  \textbf{Case 2.}  Suppose that $\mc{P} = \langle \mc{A} \, | \, \mc{R},
  s \rangle$ and $\mc{Q} = \langle \mc{A} \, | \, \mc{R} \rangle$ where
  $s \in \fm{A}$ is null-homotopic over $\mc{Q}$.  Let $(x_i, r_i)$
  be a null-$\mc{Q}$-expression for $s$ with area $M$ and radius
  $K$.  If $w \in \fm{A}$ is a null-homotopic word of length at most
  $n$ then there exists a null-$\mc{P}$-expression $\Sigma = (y_i,
  z_i)_{i=1}^L$ for $w$ with area $L \leq \alpha(n)$ and radius at
  most $\rho(n)$.  Substituting $\prod_{i=1}^M x_i r_i x_i^{-1}$ for
  each occurrence of $s$ in the product $\prod_{i=1}^L y_i z_i
  y_i^{-1}$ gives a product which is freely equal to $w$ in
  $F(\mc{A})$.  The corresponding null-$\mc{Q}$-expression has area
  at most $ML$ and radius at most $\rho(n) + K$.  Thus $(L
  \alpha(n), \rho(n) +K)$ is an area-radius pair for
  $\mc{Q}$.

  \textbf{Case 3.}  Suppose that $\mc{P} = \langle \mc{A} \, | \mc{R}
  \rangle$ and $\mc{Q} = \langle \mc{A}, b \, | \, \mc{R}, bu_b^{-1}
  \rangle$ where $u_b \in \fm{A}$ and $b u_b^{-1}$ is null-homotopic
  over $\mc{P}$.  Define $K = |u_b|$.  Suppose $w \in (\mc{A} \cup
  \{b\})^{\pm\bast}$ is a null-homotopic word of length at most $n$;
  say $w \equiv v_0 b^{\epsilon_1} v_1 \ldots b^{\epsilon_L} v_L$
  for some $v_i \in \fm{A}$ and $\epsilon_i \in
  \{\pm1\}$.  Insert cancelling pairs $u_b^{-1} u_b$ into $w$ to
  obtain the word $w' \equiv v_0 (b u_b^{-1} u_b)^{\epsilon_1} v_1
  \ldots (b u_b^{-1} u_b)^{\epsilon_L} v_L$ with $w' \FreeEq w$.
  Define $v_0', \ldots, v_L'$ to be the words in $\fm{A}$ such that
  $w' \equiv v_0' (b u_b^{-1})^{\epsilon_1} v_1'
  \ldots (b u_b^{-1})^{\epsilon_L} v_L'$ and note that $\sum_{i=1}^L
  |v_i'| \leq K|w| \leq Kn$.  For each $i \in \{0, \ldots, L\}$
  define $\tau_i \equiv v_i' v_{i+1}' \ldots v_L'$.  Then \[w'
  \FreeEq  \tau_0 \prod_{i=1}^L \tau_i^{-1} (b u_b)^{\epsilon_i}
  \tau_i\] and $|\tau_i| \leq \sum_{i=1}^L |v_i'| \leq Kn$.  The word $\tau_0$
  is null-homotopic over
  $\mc{Q}$ and hence over $\mc{P}$ and so there exists a
  null-$\mc{P}$-expression $(x_i, r_i)_{i=1}^M$ for $\tau_0$ with area
  at most $\alpha(Kn)$ and radius at most $\rho(Kn)$.  Thus \[w
  \FreeEq \prod_{i=1}^M x_i r_i x_i^{-1} \prod_{i=1}^L \tau_i^{-1} (b
  u_b^{-1})^{\epsilon_i} \tau_i\] and so we obtain a null-$\mc{Q}$-expression
  for $w$ with area at most $M + L \leq \alpha(Kn) +
  n$ and radius at most $\max \{ \max_i |x_i|, \max_i|v_i'| \} \leq
  \max\{\rho(Kn), Kn\} \leq \rho(Kn) + Kn$.  Thus $(\alpha(Kn) + n,
  \rho(Kn) + Kn)$ is an area-radius pair for $\mc{Q}$.

  \textbf{Case 4.}  Suppose that $\mc{P} = \langle \mc{A}, b \, | \, \mc{R}, bu_b^{-1}
  \rangle$ and $\mc{Q} = \langle \mc{A} \, | \mc{R}
  \rangle$ where $u_b \in \fm{A}$ and $b u_b^{-1}$ is null-homotopic
  over $\mc{Q}$.  Define $K = |u_b|$.  Consider the retraction $\pi : (\mc{A} \cup
  \{b\})^{\pm\bast} \rightarrow \fm{A}$ which is the identity on
  $\mc{A}$ and maps $b^{\pm1} \mapsto u_b^{\pm1}$.  Note that $\pi$ induces a
  retraction $F(\mc{A} \cup \{b\}) \rightarrow F(\mc{A})$.  Suppose $w
  \in \fm{A}$ is a null-homotopic word of length at most $n$ and let
  $(x_i, z_i)_{i=1}^M$ be a null-$\mc{P}$-expression for $w$ with
  area at most $\alpha(n)$ and radius at most $\rho(n)$.  Let $S$ be
  the subset of $\{1, \ldots, m\}$ consisting of those $i$ for which
  $z_i \in \mc{R}^{\pm1}$.  Then $(\pi(x_i), \pi(z_i))_{i \in S}$ is a
  null-$\mc{Q}$-expression for $w$ with area at most $M$ and radius
  at most $K \rho(n)$.  Thus $(\alpha(n), K\rho(n))$ is an area-radius pair for $\mc{Q}$.
\end{proof}

\subsection{Changing Between Infinite Presentations} \label{sec2}

Up to this point, all the definitions of this section have been
standard; we now introduce something new.  We saw above that the
Dehn functions of all finite presentations of a fixed group have the
same asymptotic behaviour.  This is not true, however, for
presentations with an infinite number of relators, where the
behaviour of the Dehn function may vary markedly.  Indeed, for any
group if we take the set of relators to be the set of all
null-homotopic words then we obtain a presentations whose Dehn
function is constant.  In order to regain some control over how the
Dehn function changes when changing between (possibly non-finite)
presentations, we introduce the following notions.

\begin{defn} \label{def2}
  An \emph{index} on a set $\mc{X}$ is a function $\| \cdot \| :
  \mc{X} \rightarrow \N$.  This is extended to an index on the set
  $\mc{X}^{\pm1}$ by setting $\| x^{-1} \| = \| x \|$.  An
  \emph{indexed presentation} is a pair $(\mc{P}, \| \cdot \|)$
  where $\mc{P} = \langle \mc{A} \, | \, \mc{R} \rangle$ is a presentation
  and $\| \cdot \|$ is an index on $\mc{R}$.

  Let $(\mc{P}, \| \cdot \|)$ be an indexed presentation whose set of
  generators $\mc{A}$ is finite.  A pair $(\alpha, \pi)$ of
  functions $\alpha, \pi : \N \rightarrow \N$ is said to be an \emph{
  area-penetration pair} for $(\mc{P}, \| \cdot \|)$ if for all null-homotopic words
  $w \in \mc{A}^{\pm\bast}$ with $|w| \leq n$ there exists a
  null-$\mc{P}$-expression $(x_i, r_i)_{i=1}^m$ for $w$ with area
  $m \leq \alpha(n)$ and with $\|r_i\| \leq \pi(n)$ for each $i$.

  Given $\mc{X} \subseteq \fm{A}$ write $\llangle \mc{X} \rrangle$ for the
  normal closure of the image of $\mc{X}$ in $F(\mc{A})$.  Let $\mc{S}
  \subseteq \fm{A}$ be a set of words with $\llangle \mc{S} \rrangle = \llangle
  \mc{R} \rrangle$.  Then $\mc{Q} = \langle \mc{A} \, | \, \mc{S} \rangle$ presents
  the same group as $\mc{P}$.  The \emph{relational area
  function} of $(\mc{P}, \| \cdot \|)$ over
  $\mc{Q}$ is defined to be the function $\N \rightarrow \N \cup \{\infty\}$ given
  by \[\RArea (n) = \max \{ \Area_\mc{Q} (r) \,
  : \, r \in \mc{R}, \|r\| \leq n \}.\]
\end{defn}

\begin{prop} \label{prop5}
  Let $(\mc{P}, \| \cdot \|)$ and $\mc{Q}$ be as in Definition~\ref{def2}.  Let $(\alpha, \pi)$ be an area-penetration
  pair for $(\mc{P}, \| \cdot \|)$ and let $\RArea$ be the relational
  area function of $(\mc{P}, \| \cdot \|)$ over
  $\mc{Q}$.  Then the Dehn function $\delta_\mc{Q}$ of the
  presentation $\mc{Q}$ satisfies \[\delta_\mc{Q}(n) \leq
  \alpha(n) \RArea(\pi(n)).\]
\end{prop}

Since the proof of this result is straightforward we omit it.

\section{Isoperimetric Functions for Cyclic Extensions}

Let $K \lhd \Gamma$ be a pair of finitely presented groups with
$\Gamma / K \cong \Z$.  In this section we show how a presentation
$\mc{P}_\Gamma$ of $\Gamma$ gives rise to an infinite presentation
$\mc{P}_K^\infty$ for $K$.  The relators of $\mc{P}_K^\infty$ come
equipped with an index $\| \cdot \|$ and we prove that an
area-radius pair for $\mc{P}_\Gamma$ is actually an area-penetration
pair for $(\mc{P}_K^\infty, \| \cdot \|)$.

Let $\mc{A}$ be a finite generating set for $K$.  Choose an element
$t$ of $\Gamma$ whose image generates $\Gamma / K \cong \Z$ and let
$\theta$ be the automorphism of $K$ induced by conjugation by $t$.
Let $\mathcal{P}_K = \langle \mathcal{A} \, | \, \mathcal{R}
\rangle$ be a presentation for $K$ and for each $a \in \mathcal{A}$
let $w_a$ be a word in $\mathcal{A}^{\pm\bast}$ representing
$\theta(a)$.  Define $\mathcal{S}$ to be the set of words $\{
tat^{-1}w_a^{-1} \, : \, a \in \mathcal{A} \}$ and let
$\mc{P}_\Gamma$ be the presentation $\langle \mc{A}, t \, | \,
\mc{R}, \mc{S} \rangle$ of $\Gamma$.

For each $k \in \Z$, let $\Phi_k : \mc{A}^{\pm\bast} \rightarrow
\mc{A}^{\pm\bast}$ be an endomorphism lifting $\theta^k : K
\rightarrow K$ which commutes with the inversion involution of
$\mc{A}^{\pm\bast}$.  We take $\Phi_0$ to be the identity.  Define
the following collections of words in $\fm{A}$:
\begin{align*}\overline{\mc{R}} &= \{ \Phi_k(r) \, : \, r \in \mc{R}, k \in
\Z \}\\
\overline{\mc{S}} &= \{ \Phi_{k+1}(a) \Phi_k(w_a)^{-1} \, : \, a \in
\mc{A}, k \in \Z\}.\end{align*}  Note that each word in
$\overline{\mc{R}} \cup \overline{\mc{S}}$ represents the identity
in $K$.  Since $\mc{R} \subseteq \overline{\mc{R}}$, the
presentation $\mc{P}_K^\infty = \langle \mc{A} \, | \, \ol{\mc{R}},
\ol{\mc{S}} \rangle$ presents $K$. Define an index $\| \cdot \|$ on
$\overline{\mc{R}} \cup \overline{\mc{S}}$ by setting $\|\omega \|$
to be the minimal value of $|k|$ such that either $\omega \equiv
\Phi_k(r)$ for some $r \in \mc{R}$ or $\omega \equiv \Phi_{k+1}(a)
\Phi_k(w_a)^{-1}$ for some $a \in \mc{A}$.

The following theorem is the principal result of this section.  The
reader may find it instructive to translate the given proof into the
language of either van Kampen diagrams (see, e.g., \cite{brid02}) or
pictures (see, e.g., \cite{Fenn1}) where the ideas involved are
perhaps more intuitive.

\begin{thm} \label{thm2}
  If $(\alpha, \rho)$ is an area-radius pair for
  $\mc{P}_\Gamma$ then it is also an area-penetration pair
  for the indexed presentation $(\mc{P}_K^\infty, \| \cdot \|)$.
\end{thm}

\begin{proof}
  Let $w \in \fm{A}$ be a null-homotopic word of length at most $n$
  and let $(x_i, z_i)_{i=1}^m$ be a null-$\mc{P}_\Gamma$-expression
  for $w$ with $m \leq \alpha(n)$ and with $|x_i| \leq \rho(n)$ for
  each $i$.

  We write $h(u)$ for the exponent sum in the letter $t$ of a word
  $u \in (\mc{A} \cup \{t\})^{\pm\bast}$ and define $\widetilde{N}$
  to be the submonoid of $(\mc{A} \cup \{t\})^{\pm\bast}$ consisting
  of all those words $u$ with $h(u) = 0$.  Define $\mc{X}$ to be the set of
  words $\{t^k a t^{-k} \, : \, a \in \mc{A}, k \in \Z \} \leq (\mc{A} \cup
  \{t\})^{\pm\bast}$.  Let $L$
  be the submonoid of $\widetilde{N}$ generated by $\mc{X}^{\pm1}$
  and note that $L$ is free on this basis.
  If $u \in \widetilde{N}$ write
  $\Lambda(u)$ for the unique word in $L$ which is freely equal to $u$ in
  $F(\mc{A} \cup \{t\})$ and freely reduced as an element of $F(\mc{X})$.  For
  each $i \in \{1, \ldots, m\}$, define
  $\ol{x}_i \equiv \Lambda(x_i t^{-h(x_i)})$ and $\ol{z}_i =
  \Lambda(t^{h(x_i)} z_i t^{-h(x_i)})$.  Define $\sigma \equiv
  \prod_{i=1}^m \ol{x}_i \ol{z}_i \ol{x}_i^{-1}$ and note that $w
  \FreeEq \sigma$ in $F(\mc{A} \cup \{t\})$.

  Define a homomorphism $\Psi : L \rightarrow \fm{A}$, which
  commutes with the inversion involution of $L$, by mapping $t^k a t^{-k}
  \mapsto \Phi_k(a)$.  Let $N$ be the kernel of the homomorphism
  $F(\mc{A} \cup \{t\}) \rightarrow \Z$ defined by mapping $t$ to $1$ and each
  $a \in \mc{A}$ to $0$, and note that $N$ is free with basis the image of
  $\mc{X}$.  Thus $\Psi$ descends to a homomorphism $N \rightarrow
  F(\mc{A})$ and since $w \FreeEq \sigma$ in $N$ we have $\Psi(w) \FreeEq
  \Psi(\sigma)$ in $F(\mc{A})$.  Observe that $\Psi(\sigma) \equiv
  \prod_{i=1}^m \Psi(\ol{x}_i) \Psi(\ol{z}_i) \Psi(\ol{x}_i)^{-1}$
  and $\Psi(w) \equiv w$ since $w$ contains no occurrence of the letter $t$.

  If $z_i \equiv a_1 \ldots a_l \in \mc{R}$ then $\ol{z}_i \equiv
  t^k a_1 t^{-k} \ldots t^k a_l t^{-k}$ for some $k \in \Z$ with $|k| = |h(x_i)| \leq
  |x_i|$.  Thus $\Psi(\ol{z}_i) \equiv \Phi_k(z_i)$ where $|k| \leq
  \rho(n)$.  If $z_i \equiv t a t^{-1} a_1 \ldots a_l \in \mc{S}$
  then $\ol{z}_i \equiv t^{k+1} a t^{-k-1} t^k a_1 t^{-k} \ldots
  t^{k} a_l t^{-k}$ for some $k \in \Z$ with $|k| = |h(x_i)| \leq |x_i|$.  Thus
  $\Psi(\ol{z}_i) \equiv \Phi_{k+1}(a) \Phi_k(w_a)^{-1}$ where
  $|k| \leq \rho(n)$.  In either case we
  have $\Psi(\ol{z}_i) \in \ol{\mc{R}} \cup \ol{\mc{S}}$ and
  $\|\Psi(\ol{z}_i)\| \leq \rho(n)$.  Thus $(\Psi(\ol{x}_i),
  \Psi(\ol{z}_i))_{i=1}^m$ is a null-$\mc{P}_K^{\infty}$-expression
  for $w$ and, since $w$ was arbitrary, we see that $(\alpha, \rho)$ is
  an area-penetration pair for $\mc{P}_K^{\infty}$.
\end{proof}

\section{Proof of Theorem \ref{thm1}} \label{sec1}

Recall from the introduction that $\Delta$ is a finite, flag
simplicial complex defining a right-angled Artin group $A$ with
standard presentation $\mc{P}_A$.  The Bestvina-Brady subgroup of
$A$ is defined to be the kernel $H_\Delta$ of the homomorphism $A
\rightarrow \Z = \langle t \rangle$ which maps each of the
generators of $\mc{P}_A$ to $t$.  The group $H_\Delta$ is finitely
presented if and only if $\Delta$ is simply connected \cite{Bestvina1};
we now describe such a presentation.

Let $\Edge(\Delta)$ be the set of directed edges of $\Delta$ (so the
cardinality of $\Edge(\Delta)$ is twice the number of $1$-simplices
in $\Delta$).  We write $\iota e$ and $\tau e$ respectively for the
initial and terminal vertices of $e$ and we write $\overline{e}$ for
the edge $e$ with the opposite orientation.  We say that the
directed edges $e_1, \ldots, e_n$ form a combinatorial path in
$\Delta$, written $e_1 \cdot \ldots \cdot e_n$, if $\tau e_i = \iota
e_{i+1}$ for all $i$. If furthermore $\tau e_n = \iota e_1$ then we
say that $e_1 \cdot \ldots \cdot e_n$ is a combinatorial $1$-cycle.

In \cite{Dicks1} Dicks and Leary show that if $\Delta$ is simply
connected then $H_\Delta$ is finitely presented by $\mathcal{P}_H =
\langle \Edge(\Delta) \, | \, \mc{R}_H \rangle$ where $\mc{R}_H$
consists of all words $e \overline{e}$ for $e \in \Edge(\Delta)$ and
all words $efg$ and $e^{-1} f^{-1} g^{-1}$ where $e \cdot f \cdot g$
is a combinatorial $1$-cycle in $\Delta$.  If we identify the
vertices of $\Delta$ with the generators of $A$, then the embedding
$H_\Delta \hookrightarrow A$ is given by mapping $e \mapsto \iota e
(\tau e)^{-1}$ for each edge $e \in \Edge(\Delta)$. In this section
we will prove Theorem~\ref{thm1} by demonstrating that the Dehn
function $\delta$ of the presentation $\mc{P}_H$ satisfies
$\delta(n) \preceq n^4$.

Choose a base vertex $q$ and a spanning tree $T$ in the $1$-skeleton
of $\Delta$.  Given $n \in \Z$ and vertices $u$ and $v$ of $\Delta$
we write $p_n(u, v)$ for the element $e_1^n \ldots e_l^n$ of
$\Edge(\Delta)^{\pm\bast}$ where $e_1 \cdot \ldots \cdot e_l$ is the
unique geodesic combinatorial path in $T$ from $u$ to $v$.  We write
$p(u, v)$ as shorthand for $p_1(u, v)$.  Note that as group elements
\begin{equation} \label{Eqn1} \begin{aligned}
  p_n(u, v)^{-1} &= (e_1^n \ldots e_l^n)^{-1} \\
  &= e_l^{-n} \ldots e_1^{-n} \\
  &= \ol{e_l}^n \ldots \ol{e_1}^n \\
  &= p_n(v, u)
\end{aligned} \end{equation} in $H_\Delta$.  For each $e \in \Edge(\Delta)$
define $w_e$ to be the word $p(q, \iota e) e p(\iota e, q)$ of
$\Edge(\Delta)^{\pm\bast}$.  In \cite{Dicks1} it is proved that
mapping $e \mapsto w_e$ defines an automorphism $\theta$ of
$H_\Delta$  and that $H_\Delta \rtimes_\theta \Z$ is isomorphic to
$A$ with $e \in \Edge(\Delta)$ corresponding to $\iota e (\tau
e)^{-1}$ and the generator $t$ of $\Z$ corresponding to $q \in A$.
It is also shown that if $e_1 \cdot \ldots \cdot e_l$ is a
combinatorial $1$-cycle then $e_1^n \ldots e_l^n$ is null-homotopic
in $H_\Delta$. Define $\mc{S}_H$ to be the set of words $\{t e
t^{-1} w_e \, : \, e \in \Edge(\Delta)\}$ in $(\Edge(\Delta) \cup
\{t\})^{\pm\bast}$ so $A$ is finitely presented by $\mc{P}_A' =
\langle \Edge(\Delta) , t \, | \, \mc{R}_H, \mc{S}_H \rangle$.

It is proved in \cite{Charney1} that $A$ is CAT(0) (see
\cite{brid99} for the definition of a CAT(0) group) so by
 \cite[Proposition~III.$\Gamma$.1.6]{brid99} there exists a
finite presentation for $A$ and an area-radius pair $(\alpha, \rho)$
for this presentation with $\alpha(n) \simeq n^2$ and $\rho(n)
\simeq n$.  By Proposition~\ref{prop4} it follows that there is an
area-radius pair $(\alpha', \rho')$ for $\mc{P}'_A$ with $\alpha'(n)
\simeq n^2$ and $\rho'(n) \simeq n$.

The following lemma details some properties of the automorphism
$\theta$ of $H_\Delta$.  Of these we will only need (vii), but this
property is most easily proved via the preceding sequence of
assertions.

\renewcommand{\labelenumi}{(\roman{enumi})}
\begin{lem}
  For all $e \in \Edge(\Delta)$ and $n \in \Z$ the following
  equalities hold in $H_\Delta$:
  \begin{enumerate}
  \item $\theta(e) = p(q, \iota e) e p(q, \iota e)^{-1} = p(q, \iota e) e^2 p(\tau e, q)
  = p(q, \iota e) e^2 p(q, \tau e)^{-1}$.

  \item $\theta(e^n) = p(q, \iota e) e^n p(\iota e, q) =
  p(q, \iota e) e^{n+1} p(\tau e, q) = p(q, \iota e) e^{n+1} p(q, \tau e)^{-1}$.

  \item If $e_1 \cdot \ldots \cdot e_l$ is a combinatorial path
  then \[\theta(e_1^n \ldots e_l^n) = p(q, \iota e) e_1^{n+1}
  \ldots e_l^{n+1} p(\tau e, q).\]

  \item $\theta^{-1}(e) = p_{-1}(q, \iota e) p_{-1}(\tau e, q) = p_{-1}(q, \iota e)
  e p_{-1}(\iota e, q) = p_{-1}(q, \iota e) e p_{-1}(q, \iota e)^{-1}$.

  \item $\theta^{-1}(e^n) = p_{-1}(q, \iota e) e^n p_{-1}(\iota e,
  q) = p_{-1}(q, \iota e) e^{n-1} p_{-1}(\tau e, q) = p_{-1}(q, \iota e)
  e^{n-1} p_{-1}(q, \tau e)^{-1}$.

  \item If $e_1 \cdot \ldots \cdot e_l$ is a combinatorial path
  then \[\theta^{-1}(e_1^n \ldots e_l^n) = p_{-1}(q, \iota e) e_1^{n-1}
  \ldots e_l^{n-1} p_{-1}(\tau e, q).\]

  \item $\theta^k(e) = p_k(q, \iota e) e^{k+1} p_k(\tau e, q)$.
\end{enumerate}
\end{lem}

\begin{proof}
\mbox{}
   \begin{enumerate}
     \item Follows from equation \eqref{Eqn1} and the fact that $p(q, \iota e) e p(\tau
     e ,q)$ is null-homotopic.

     \item Follows from (i) on telescoping.

     \item Follows from (ii) on telescoping.

     \item Follows from the calculation \begin{align*}
       \theta(p_{-1}(q, \iota e) p_{-1}(\tau e, q))
       &= p(q, q) p_0(q, \iota e) p(\iota e, q) p(q, \tau e)
       p_0(\tau e, q) p(q,q) \\
       &= p(\iota e, q) p(q, \tau e) \\
       &= e.
     \end{align*}

     \item Follows from (iv) on telescoping.

     \item Follows from (v) on telescoping.

     \item Follows from (iii) and (vi) by induction on $|k|$.
   \end{enumerate}
\end{proof}

For each $n \in \Z$ define a homomorphism $\Phi_n :
\Edge(\Delta)^{\pm\bast} \rightarrow \Edge(\Delta)^{\pm\bast}$ which
commutes with the inversion involution and is a lift of $\theta^n$
by mapping $e \mapsto p_n(q, \iota e) e^{n+1} p_n(\tau e, q)$.
Define the collections of words
\begin{align*} \overline{\mc{R}}_H &=
\{\Phi_n(r) \, : \, r \in \mc{R}_H, n \in \Z\}\\
\overline{\mc{S}}_H &= \{ \Phi_{n+1}(e) \Phi_n(w_e)^{-1} \, : \, e
\in \Edge(\Delta), n \in \Z\}\end{align*} in
$\Edge(\Delta)^{\pm\bast}$, and consider the presentation
$\mc{P}_H^\infty = \langle \Edge(\Delta) \, | \,
\overline{\mc{R}}_H, \overline{\mc{S}}_H \rangle$ of $H_\Delta$.
Define an index $\| \cdot \|$ on $\overline{\mc{R}}_H \cup
\overline{\mc{S}}_H$ by setting $\| \omega \|$ to be the minimum
value of $|n|$ such that either $\omega \equiv \Phi_n(r)$ for some
$r \in \mc{R}_H$ or $\omega \equiv \Phi_{n+1}(a) \Phi_n(w_a)^{-1}$
for some $e \in \Edge(\Delta)$.

By Theorem~\ref{thm2}, $(\alpha', \rho')$ is an area-penetration
pair for the indexed presentation $(\mc{P}_H^\infty, \|\cdot\|)$.
Thus, to complete the proof of Theorem~\ref{thm1} it suffices, by
Proposition~\ref{prop5}, to show that the relational area function
$\RArea$ of $(\mc{P}_H^\infty, \| \cdot \|)$ over $\mc{P}_H$
satisfies $\RArea(n) \simeq n^2$.  We devote the remainder of the
section to this task.

Let $\Dist$ be the length metric on the $1$-skeleton of $T$ given by
setting the length of each edge to $1$.  Define \[L = \max\{\Dist(u,
v) \, : \, u, v \in \Vert(\Delta) \}.\]

\begin{lem} \label{lem4}
  $\Area_{\mc{P}_H}\big(\Phi_n(e \ol{e})\big) \leq (2L+1)|n| + 1$
  for all $e \in \Edge(\Delta)$.
\end{lem}

\begin{proof}
  The calculation (\ref{Eqn1}) shows that $p_n(q, v)^{-1}$ can be
  converted to $p_n(v, q)$ at a $\mc{P}_H$-cost of at most $L|n|$
  for all $v \in \Vert(\Delta)$.  The following is a
  null-$\mc{P}_H$-scheme for the word $\Phi_n(e
  \ol{e})$: \begin{align*}
    \Phi_n(e \ol{e}) &\equiv p_n(q, \iota e) e^{n+1} p_n(\tau e, q)
    p_n(q, \tau e) \ol{e}^{n+1} p_n(\iota e, q) &&\\
    &\rightsquigarrow p_n(q, \iota e) e^{n+1} \ol{e}^{n+1} p_n(\iota e, q)
    &&\text{Cost} \leq L|n|\\
    &\rightsquigarrow p_n(q, \iota e) p_n(\iota e, q)
    &&\text{Cost} \leq |n|+1\\
    &\rightsquigarrow \emptyset &&\text{Cost} \leq L|n|
  \end{align*} Total cost $\leq (2L + 1) |n| + 1$.
\end{proof}

\begin{lem} \label{lem1}
  Let $e \cdot f \cdot g$ be a combinatorial $1$-cycle in $\Delta$.
  Then $\Area_{\mc{P}_H}(e^n f^n g^n) \leq 3|n|^2$.
\end{lem}

\begin{proof}
  Note that the relators $efg$ and $e^{-1} f^{-1} g^{-1}$ imply
  that $ef = g^{-1} = fe$, so $[e, f]$ is null-homotopic with
  $\mc{P}_H$-$\Area$ $2$.  The following is a null-$\mc{P}_H$-scheme for
  the word $e^n f^n g^n$: \begin{align*}
    e^n f^n g^n &\rightsquigarrow e^n f^n (f^{-1} e^{-1})^n
    &&\text{Cost} \leq |n|\\
    &\rightsquigarrow e^n f^n f^{-n} e^{-n} &&\text{Cost} \leq
    2|n|^2\\
    &\stackrel{\rm free}{=} \emptyset &&
  \end{align*} Total cost $\leq 2|n|^2 + |n| \leq 3|n|^2$.
\end{proof}

\begin{lem} \label{lem5}
  Let $e \cdot f \cdot g$ be a combinatorial $1$-cycle in
  $\Delta$.  Then $\Area_{\mc{P}_H} \big( \Phi_n(efg) \big) \leq
  3|n|^2 + (3L + 6)|n| + 3$.
\end{lem}

\begin{proof}
  The following is a null-$\mc{P}_H$-scheme for the word
  $\Phi_n(efg)$: \begin{align*}
    \Phi_n(efg) &\equiv p_n(q, \iota e) e^{n+1} p_n(\tau e, q)
    p_n(q, \iota f) f^{n+1} p_n(\tau f, q)\ldots &&\\
    & \qquad \ldots p_n(q, \iota g) g^{n+1} p_n(\tau g, q) &&\\
    &\rightsquigarrow p_n(q, \iota e) e^{n+1} f^{n+1} g^{n+1} p_n(\tau g, q)
    &&\text{Cost} \leq 2L|n|\\
    &\rightsquigarrow p_n(q, \iota e) p_n(\tau g, q)
    &&\text{Cost} \leq 3|n+1|^2\\
    &\rightsquigarrow \emptyset &&\text{Cost} \leq L|n|.
  \end{align*} Total cost $\leq 3|n|^2 + (3L + 6)|n| + 3$.
\end{proof}

\begin{defn}
  Given a combinatorial $1$-cycle $C$ in $\Delta$, a sequence $(C_i)_{i=0}^m$ of combinatorial
  $1$-cycles is said to be \emph{combinatorial null-homotopy} for
  $C$ if $C_0 = C$, $C_m = \emptyset$ and each
  $C_{i+1}$ is obtained from $C_i$ by one of the following moves:
  \begin{enumerate}
    \item \emph{$1$-cell expansion}: $C_i = e_1 \cdot \ldots \cdot
    e_l \rightsquigarrow C_{i+1} = e_1 \cdot \ldots \cdot e_k
    \cdot e \cdot \ol{e} \cdot e_{k+1} \cdot \ldots \cdot e_l$ for
    some $k$, where $e \in \Edge(\Delta)$;

    \item \emph{$1$-cell collapse}: Reverse of a $1$-cell
    expansion;

    \item \emph{$2$-cell expansion}: $C_i = e_1 \cdot \ldots \cdot
    e_l \rightsquigarrow C_{i+1} = e_1 \cdot \ldots \cdot e_k
    \cdot e \cdot f \cdot g \cdot e_{k+1} \cdot \ldots \cdot e_l$
    for some $k$, where $e \cdot f \cdot g$ is a combinatorial
    $1$-cycle;

    \item \emph{$2$-cell collapse}: Reverse of a $2$-cell
    expansion.
  \end{enumerate}
\end{defn}

\begin{lem} \label{lem2}
  If $(C_i)_{i=0}^m$ is a combinatorial null-homotopy for the
  $1$-cycle $e_1 \cdot \ldots \cdot e_l$ then the word $e_1^n
  \ldots e_l^n$ has $\mc{P}_H$-Area $\leq 3m|n|^2$.
\end{lem}

\begin{proof}
  Given a combinatorial $1$-cycle $C = e_1 \cdot \ldots \cdot
  e_l$, write $W_n(C)$ for the word $e_1^n \ldots e_l^n \in
  \Edge(\Delta)^{\pm\bast}$.  If the $1$-cycle $C_i$ is obtained
  from $C_{i-1}$ by a $1$-cell expansion or collapse then, by
  repeated application of a relator $e \ol{e}$, the word
  $W_n(C_{i-1})$ can be converted to the word $W_n(C_i)$ at a
  $\mc{P}_H$-cost of at most $|n|$.  If the $1$-cycle $C_i$ is
  obtained from $C_{i-1}$ by a $2$-cell expansion or collapse then,
  by Lemma~\ref{lem1}, the word $W_n(C_{i-1})$ can be converted to
  the word $W_n(C_i)$ at a $\mc{P}_H$-cost of at most $3|n|^2$.

  Define $m_1$ to be the number of $i$ for which $C_i$ is obtained
  from $C_{i-1}$ by a $1$-cell expansion or collapse.  Define $m_2$
  to be the number of $i$ for which $C_i$ is obtained
  from $C_{i-1}$ by a $2$-cell expansion or collapse.  Then the
  $\mc{P}_H$-Area of $e_1^n \ldots e_l^n = W_n(C)$ is at most
  $m_1|n| + 3m_2|n|^2 \leq 3(m_1 + m_2) |n|^2 = 3m|n|^2$.
\end{proof}

\begin{lem} \label{lem3}
  There exists a constant $K$ such that $\Area_{\mc{P}_H}\big(
  p_n(q, \iota e) e^n p_n(\tau e, q) \big) \leq K|n|^2$ for all $e
  \in \Edge(\Delta)$.
\end{lem}

\begin{proof}
  Given $e \in \Edge(\Delta)$ write $\gamma_\iota(e)$ and
  $\gamma_\tau(e)$ respectively for the unique combinatorial
  geodesic paths in $T$ from $q$ to $\iota e$ and from $\tau e$ to
  $q$.  Then $\gamma_\iota(e) \cdot e \cdot \gamma_\tau(e)$ is a
  combinatorial $1$-cycle for which there exists a combinatorial
  null-homotopy $\big(C_i(e)\big)_{i=0}^{m(e)}$ since $\Delta$ is simply-connected.
  By Lemma~\ref{lem2} $\Area_{\mc{P}_H} \big( p_n(q, \iota e) e^n p_n(\tau e, q)
  \big) \leq 3m(e)|n|^2$, so we can take $K = 3\max\{ m(e) \, : \,
  e \in \Edge(\Delta)\}$.
\end{proof}

\begin{lem} \label{lem6}
  Let $e \cdot f \cdot g$ be a combinatorial $1$-cycle in
  $\Delta$.  Then $\Area_{\mc{P}_H} \big(
  \Phi_n(e^{-1}f^{-1}g^{-1}) \big) \leq (3K + 4) |n|^2 + (6L + 6) |n| + 5$,
  where $K$ is the constant from Lemma~\ref{lem3}.
\end{lem}

\begin{proof}
  The following is a null-$\mc{P}_H$-scheme for the word $\Phi_n(e^{-1} f^{-1}
  g^{-1})$: \begin{align*}
    \Phi_n(e^{-1} &f^{-1} g^{-1}) &&\\
    &\equiv p_n(\tau e, q)^{-1} e^{-n-1} p_n(q, \iota e)^{-1}
    p_n(\tau f, q)^{-1} f^{-n-1} \ldots &&\\
    &\qquad \ldots
    p_n(q, \iota f)^{-1} p_n(\tau g, q)^{-1} g^{-n-1} p_n(q, \iota g)^{-1} &&\\
    &\rightsquigarrow p_n(q, \tau e) e^{-n-1} p_n(\iota e, q)
    p_n(q, \tau f) f^{-n-1} p_n(\iota f, q) \ldots &&\\
    &\qquad \ldots
    p_n(q, \tau g) g^{-n-1} p_n(\iota g, q) &&\text{Cost} \leq 6L|n|\\
    &\stackrel{\rm free}{=}
    p_n(q, \iota f) e^{-n-1} p_n(\tau g, q)
    p_n(q, \iota g) f^{-n-1} p_n(\tau e, q) &&\\
    &\qquad \ldots
    p_n(q, \iota e) g^{-n-1} p_n(\tau f, q)
    p_n(q , \iota f) p_n(q, \iota f)^{-1} &&\\
    &\rightsquigarrow p_n(q, \iota f) e^{-n-1} g^{-n} f^{-n-1}
    e^{-n} g^{-n-1} f^{-n} p_n(q, \iota f)^{-1} &&\text{Cost}
    \leq 3K|n|^2\\
    &\rightsquigarrow p_n(q, \iota f) e^{-n-1} (ef)^n f^{-n-1}
    e^{-n} (ef)^{n+1} f^{-n} p_n(q, \iota f)^{-1} &&\text{Cost}
    \leq 2|n| + 1\\
    &\rightsquigarrow p_n(q, \iota f) e^{-n-1} e^n f^n f^{-n-1}
    e^{-n} e^{n+1} f^{n+1} f^{-n} p_n(q, \iota f)^{-1} &&\text{Cost}
    \leq 2|n|^2 + 2|n+1|^2\\
    &\stackrel{\rm free}{=} p_n(q, \iota f) e^{-1} f^{-1} e f
    p_n(q, \iota f)^{-1} &&\\
    &\rightsquigarrow p_n(q, \iota f) g g^{-1}
    p_n(q, \iota f)^{-1} &&\text{Cost} \leq 2\\
    &\stackrel{\rm free}{=} \emptyset.
  \end{align*}  Total cost $\leq (3K + 4) |n|^2 + (6L + 6) |n| + 5$.
\end{proof}

\begin{lem} \label{lem7}
  $\Area_{\mc{P}_H} \big( \Phi_{n+1}(e) \Phi_n(w_e)^{-1} \big)
  \leq 2K|n|^2 + (3L^2 + 2L + 2K)|n| + L + K$ for all $e \in \Edge(\Delta)$, where $K$ is
  the constant from Lemma~\ref{lem3}.
\end{lem}

\begin{proof}
  Note that if $e_1 \cdot \ldots \cdot e_l$ is a combinatorial edge-path
  in $\Delta$ then $\Phi_n(e_1 \ldots e_l) = \prod_{i=1}^l p_n(q,
  \iota e_i) e_i^{n+1} p_n(\tau e_i, q)$ can be converted to
  \[\prod_{i=1}^l p_n(q, \iota e_i) e_i^{n+1} p_n(q, \tau e_i)^{-1}
  \stackrel{\rm free}{=} p_n(q, \iota e_1) e_1^{n+1} \ldots
  e_l^{n+1} p_n(q, \tau e_l)^{-1}\] at a $\mc{P}_H$-cost of at
  most $lL|n|$.  It follows that for all $u, v \in \Vert(\Delta)$ the
  word $\Phi_n \big( p(u, v) \big)$ can
  be converted to the word $p_n(q, u) p_{n+1}(u, v) p_n(q, v)^{-1}$ at
  a $\mc{P}_H$-cost of at most $L^2 |n|$.

  The following is a null-$\mc{P}_H$-scheme for the word $\Phi_{n+1}(e)
  \Phi_n(w_e)^{-1}$: \begin{align*}
    \Phi_{n+1}(e) &\Phi_n(w_e)^{-1} &&\\
    &\equiv p_{n+1}(q, \iota e)
    e^{n+2} p_{n+1}(\tau e, q) \left[\Phi_n \big( p(q, \iota e) e
    p(\iota e, q) \big) \right]^{-1} &&\\
    &\rightsquigarrow p_{n+1}(q, \iota e)
    e^{n+2} p_{n+1}(\tau e, q) \big[ p_{n+1}(q, \iota e) p_n(q,
    \iota e)^{-1}\ldots &&\\
    &\qquad \ldots p_n(q, \iota e) e^{n+1} p_n(\tau e, q) p_n(q,
    \iota e) p_{n+1}(\iota e, q) \big]^{-1} &&\text{Cost} \leq
    2L^2|n|\\
    &\stackrel{\rm free}{=} p_{n+1}(q, \iota e)
    e^{n+2} p_{n+1}(\tau e, q) p_{n+1}(\iota e, q)^{-1} \ldots &&\\
    &\qquad \ldots p_n(q, \iota
    e)^{-1} p_n(\tau e, q)^{-1} e^{-n-1} p_{n+1}(q, \iota e)^{-1} &&\\
    &\rightsquigarrow p_{n+1}(q, \iota e) e^{n+2} p_{n+1}(\tau e, q)
    p_{n+1}(q, \iota e) \ldots &&\\
    &\qquad \ldots p_n(q, \iota e)^{-1} p_n(\tau e, q)^{-1} e^{-n-1}
    p_{n+1}(q, \iota e)^{-1} &&\text{Cost} \leq L|n+1|\\
    &\rightsquigarrow p_{n+1}(q, \iota e) e^{n+2} e^{-n-1} e^n
    e^{-n-1} p_{n+1}(q, \iota e)^{-1} &&\text{Cost} \leq K|n+1|^2 +
    K |n|^2\\
    &\stackrel{\rm free}{=} \emptyset.
  \end{align*} Total cost $\leq 2K|n|^2 + (2L^2 + L + 2K)|n| + L + K$.
\end{proof}

Combining Lemmas~\ref{lem4}, \ref{lem5}, \ref{lem6} and \ref{lem7}
we see that
\[\RArea(n) \leq (3K+4) n^2 + (6L^2 + 2K + 6) n + L + K + 5.\]  This completes the
proof of Theorem~\ref{thm1}.

\section{A Bestvina-Brady Group with Quartic Dehn Function}

In Section~2.5.2 of \cite{Brady1} Brady gives a sequence $(K_m)_{m
\in \N}$ of finite, flag simplicial complexes and suggests that the
Bestvina-Brady group associated to $K_m$ will have Dehn function
$\delta(n) \simeq n^{m+2}$.  Theorem \ref{thm1} shows that this
cannot be the case.  However, the construction does work in the
cases $m=1$ and $m=2$ and the example $K_2$ thus shows that the
bound obtained in Theorem~\ref{thm1} cannot be improved in general.
We briefly recount that example here.

\begin{figure}[h]
  \psfrag{a}{\emph{a}}
  \psfrag{b}{$b$}
  \psfrag{c}{$c$}
  \psfrag{d}{$d$}
  \centering \mbox{\includegraphics*{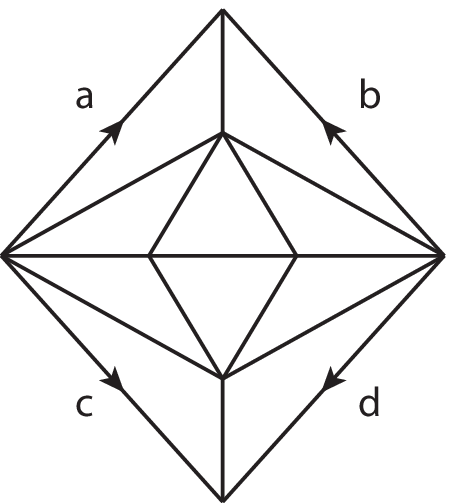}}  \caption{} \label{fig1}
\end{figure}

The complex $K_2$ is the triangulation of the disc shown in Figure~\ref{fig1}.  Let $H_{K_2}$ and $A_{K_2}$ be the Bestvina-Brady and
right-angled Artin groups respectively associated to $K_2$.  Choose
an orientation of the edges of $K_2$ so as the four edges labelled in
the figure are orientated as indicated. Let $\overline{\Edge}(K_2)
\leq \Edge(K_2)$ be the index $2$ subgroup consisting of the
positively orientated edges. Let $\mc{P}_H$ be the Dicks-Leary
presentation for $H_{K_2}$ with generating set $\Edge(K_2$), as
described in Section~\ref{sec1}. Derive from $\mc{P}_H$ the
presentation $\mc{Q}_H$ for $H$ with generating set
$\overline{\Edge}(K_2)$ by using Tietze transformations to remove
all the superfluous generators $\Edge(K_2) \smallsetminus
\overline{\Edge}(K_2)$ and all the superfluous relators $\{e
\overline{e} \, : \, e \in \Edge(K_2) \}$.

For each $k \in \N$ define $w_k \in \overline{Edge}(K_2)^{\pm\bast}$
to be the null-homotopic word \\$(da)^k (b^{-1} c^{-1})^k (ad)^k
(c^{-1} b^{-1})^k$, where $a$, $b$, $c$ and $d$ are the orientated
edges labelled in the figure.  In \cite{Brady1} Brady describes how
to construct a van Kampen diagram $\Omega_k$ over the presentation
$\mc{Q}_H$ with boundary label $w_k$ and $\Area(\Omega_k) \simeq
k^4$.  It is shown that the presentation $2$-complex $\Sigma$
associated to $\mc{Q}_H$ is aspherical and that the diagram
$\Omega_k$ embeds in the universal cover of $\Sigma$. It follows
that $\Area(w_k) = \Area(\Omega_k)$ and hence that the Dehn function
of $\mc{Q}_H$ is $\simeq n^4$.

\bigskip

\emph{Acknowledgements.} I would like to thank my thesis advisor, Martin Bridson, for his many helpful comments made during the preparation of this article.

\bibliography{BBGroups.bbl}

\end{document}